\documentclass{IEEEtran}
\usepackage{amsmath,amssymb,amsfonts}
\usepackage{amsthm}
\usepackage{graphicx}
\usepackage{hyperref}
\usepackage{physics}
\usepackage{bm}
\usepackage{algorithm}
\usepackage{algpseudocodex}
\usepackage{siunitx}
\usepackage{booktabs}
\usepackage[font=footnotesize]{caption}
\usepackage{subcaption}
\usepackage{mathtools}
\usepackage{url}
\def\BibTeX{{\rm B\kern-.05em{\sc i\kern-.025em b}\kern-.08em
    T\kern-.1667em\lower.7ex\hbox{E}\kern-.125emX}}
\newcommand{\R}{\mathbb{R}}
\newcommand{\N}{\mathbb{N}}

\newcommand{\E}{\mathbb{E}}

\newcommand{\cov}[1]{\mathrm{Cov}[#1 ]}
\let\P\relax
\newcommand{\P}[1]{\mathbb{P}\left[#1\right]}
\let\E\relax
\newcommand{\E}[1]{\mathbb{E}\left[#1\right]}

\newcommand{\suchthat}{\mathrm{s.t.}\quad}

\newcommand{\diag}[1]{\mathrm{diag}(#1)}
\DeclareMathOperator*{\minimize}{minimize}

\newtheorem{theorem}{Theorem}
\newtheorem{assumption}{Assumption}

\newtheorem{remark}{Remark}
\newtheorem{lemma}{Lemma}
\newtheorem{proposition}{Proposition}

\usepackage{cleveref}
\crefname{align}{}{}
\crefname{equation}{}{}
\crefname{figure}{Fig.}{Figs.}
\crefname{table}{Table}{Tables}
\crefname{theorem}{Theorem}{Theorems}
\crefname{definition}{Definition}{Definitions}
\crefname{lemma}{Lemma}{Lemmas}
\crefname{remark}{Remark}{Remarks}
\crefname{assumption}{Assumption}{Assumptions}
\crefname{proof}{Proof}{Proofs}
\crefname{algorithm}{Algorithm}{Algorithms}
\crefname{problem}{Problem}{Problems}
\crefname{proposition}{Proposition}{Propositions}
\crefname{corollary}{Corollary}{Corollaries}
\crefname{section}{Section}{Sections}

\allowdisplaybreaks[1] 

\def\showChanges{0} 

\newcommand{\highlight}[1]{%
    \ifnum\showChanges=1
        \textcolor{red}{#1}%
    \else
        #1%
    \fi
}
\newcommand{\mathhl}[1]{
    \ifnum\showChanges=1
        \mathcolor{red}{#1}
    \else
        #1
    \fi
}

\usepackage{marginnote}
\usepackage{tcolorbox}
\newcommand{\margincomment}[1]{%
    \ifnum\showChanges=1
    \textcolor{blue}{\setlength{\baselineskip}{10pt}\footnotesize\raggedright\marginnote{#1}}%
    \else
    \fi
}

\graphicspath{{./figures/}}

\title{Square Root-Factorized Covariance Steering}
\author{Naoya Kumagai, \IEEEmembership{Graduate Student Member, IEEE}, Kenshiro Oguri, \IEEEmembership{Member, IEEE}
\thanks{This material is based upon work supported by the Air Force Office of Scientific Research under award number FA9550-23-10512. N. K. acknowledges support for his graduate studies from the Shigeta Education Fund. The authors are with the School of Aeronautics and Astronautics, Purdue University, West Lafayette, IN 47907 USA (\{nkumagai, koguri\}@purdue.edu).}}

\begin{document}

\maketitle

\begin{abstract}
Covariance steering (CS) synthesizes a control policy which drives the state's mean and covariance matrix towards desired values. 
Offering tractable computation of a closed-loop policy which can obey chance constraints in uncertain environments, application to many real-world control problems have been proposed.
We consider the chance-constrained, discrete-time, linear time-varying CS with Gaussian noise. The contribution of this paper is a novel solution method for this problem, explicitly writing the propagation equations of the Cholesky factor of the state covariance matrix by using the QR decomposition. 
The use of the square-root form of covariance matrices brings two key benefits over other existing methods: (i) computational scalability and (ii) numerical reliability.  (i) Compared to solution methods that require large block matrix formulations, the proposed method scales better with the growth in horizon length, shows better optimality, and uses memoryless state feedback. 
(ii) Compared to another class of methods that explicitly define the covariance matrix as variables, the proposed method allows flexible cost formulations and shows better numerical reliability when uncertainty terms are smaller than the mean. 
On the other hand, these benefits come with a minor drawback: the propagation equation of covariance square roots is non-convex, necessitating sequential convex programming to solve.  However, this paper proves the global optimality of the proposed approach for CS without chance constraints. When chance constraints are present, the existing optimal CS formulation is also non-convex, and we prove that the proposed approach shares the same local minima. We verify the mathematical arguments via extensive numerical simulations.
\end{abstract}


\section{Introduction}

Uncertainty is the fundamental reason for closed-loop control design. To explicitly and probabilistically address uncertainty, covariance steering (CS) \cite{hotzCovarianceControlTheory1987,chenOptimalSteeringLinear2016} has emerged as a promising control framework. Unlike traditional methods that treat uncertainty as a disturbance to be rejected, CS defines the system state to include both the mean and the covariance matrix within the optimal control formulation. This approach yields optimal feedforward and feedback control policies that steer the system's statistical moments to, or within, desired terminal values.

However, CS possesses unique characteristics distinct from standard optimal control problems. Due to the inherent objective of steering uncertainty to small values while minimizing control effort, the variance of the state can show significantly different magnitudes throughout the time horizon. This contrasts with nominal trajectory optimization, where the magnitude of the state variables typically remains within a consistent range.
This issue is amplified in real-world applications by two factors. First, the magnitude of the uncertainty is often much smaller than that of the mean state; for example, the mean trajectory of a spacecraft can span tens of thousands of kilometers, while its position uncertainty can be on the order of kilometers \cite{kumagaiRobustCislunarLowThrust2025}. Second, safety requirements and actuator limitations necessitate the use of chance constraints (CCs). CCs couple the problem of finding the optimal feedforward sequence with that of finding the optimal feedback sequence  \cite{okamotoOptimalCovarianceControl2018}. Consequently, the optimization problem is inherently ill-scaled. Implemented on finite-precision computers, numerical optimization algorithms may struggle to converge, as we demonstrate in this paper.

Recent works on finite-horizon discrete-time CS mainly focus on converting the problem into a convex optimization problem. Skaf and Boyd \cite{skafDesignAffineControllers2010} consider an affine \textit{causal} state feedback controller and formulates the problem as a semidefinite programming (SDP) problem. Okamoto et al. \cite{okamotoOptimalCovarianceControl2018} add affine CCs while maintaining convexity. However, these methods suffer from the rapidly growing computational time as horizon length increases, due to the large linear matrix inequality (LMI) constraints involved \cite{rapakouliasDiscreteTimeOptimalCovariance2023}; see \cref{tab:comparison}. Okamoto and Tsiotras \cite{okamotoOptimalStochasticVehicle2019} remedy the need for block lower triangular (LT) matrix variables employed in \cite{skafDesignAffineControllers2010,okamotoOptimalCovarianceControl2018} by introducing an alternative affine feedback control, requiring only a block diagonal structure. Nevertheless, due to the implicit representation of the covariance matrix variables as a function of all the previous system matrices, the LMIs to be solved remain rather large. For later discussion, it is also important to note that this implicit representation in \cite{skafDesignAffineControllers2010,okamotoOptimalCovarianceControl2018,okamotoOptimalStochasticVehicle2019} allows an affine representation of the Cholesky factor of the covariance matrix in terms of the optimization variables. This in turn allows affine CCs to be expressed exactly as second-order cone constraints. (See (42) and Proposition 4 of \cite{okamotoOptimalCovarianceControl2018} to observe that \((\cdot)^{1/2} \) appears in the deterministic formulation of the constraints.)

More recently, Liu et al. \cite{liuOptimalCovarianceSteering2025} made notable breakthroughs. First, it shows the optimality of the affine \textit{memoryless} state feedback control law for the unconstrained CS problem, i.e. when CCs are not present. Second, it relaxes the nonconvex problem to an SDP, and shows that its first-order optimality conditions imply that the relaxation is tight. Compared to the block-matrix approaches (\cite{skafDesignAffineControllers2010,okamotoOptimalCovarianceControl2018,okamotoOptimalStochasticVehicle2019}), this method explicitly defines the covariance matrix variables and scales better with the growth in problem size (see \cref{tab:comparison}). Its proof of convexification has been extended to the chance-constrained problem \cite{rapakouliasDiscreteTimeOptimalCovariance2023}. 
At the same time, we have observed several drawbacks of \cite{liuOptimalCovarianceSteering2025} in application, which motivated this work: 
\begin{itemize}
\item (\textbf{Limitations of Convexification}) The proof of convexification in \cite{liuOptimalCovarianceSteering2025} requires that the partial of the Lagrangian w.r.t. the control covariance matrix is positive definite (PD) (such as in the expectation-of-quadratic case). For general objective functions, theoretical work is required to extend the proof; for example, when using a non-quadratic objective function \cite{kumagaiRobustCislunarLowThrust2025}, a small quadratic term was added to the objective function for convexification.
	
\item (\textbf{Numerical Reliability}) In scenarios where the state variance is small compared to the mean values \cite{oguriChanceConstrainedControlSafe2024a,kumagaiRobustCislunarLowThrust2025}, state-of-the-art SDP solvers struggle to converge. We attribute the numerical difficulties to dealing with the covariance matrix directly; this necessarily squares the system units and amplifies ill-scaling, causing issues for solvers working in finite precision. Furthermore, the convexification process is only valid up to the numerical accuracy of the LMI that is solved. We have observed that when the solver terminates due to numerical issues, the covariance propagation equations can be lossy.

\item (\textbf{Non-convexity under Chance Constraints}) 
It is common to encounter objective functions or CCs that are represented as convex in the square root form, e.g. quantile objective functions or affine CCs \cite{okamotoOptimalCovarianceControl2018}. Due to the square root operation, these constraints are nonconvex in the covariance matrix.
Thus, a reference covariance matrix for conservatively approximating the square root via linearization \cite{rapakouliasDiscreteTimeOptimalCovariance2023} or iterative algorithms \cite{pilipovskyComputationallyEfficientChance2024} are required.
\end{itemize}

\cref{tab:comparison} summarizes existing solution methods for discrete-time chance-constrained CS (cc-CS).
\begin{table}[t]
\caption{Comparison of existing methods for CS; \textbf{bold} indicates the more favorable property. $n$ is state size; $N$ is horizon length.}
\centering
\begin{tabular}{lcc}
\toprule
& \cite{skafDesignAffineControllers2010,okamotoOptimalCovarianceControl2018,okamotoOptimalStochasticVehicle2019}& \cite{liuOptimalCovarianceSteering2025,rapakouliasDiscreteTimeOptimalCovariance2023,
pilipovskyComputationallyEfficientChance2024}   \\
\midrule
Covariance & Implicit & Explicit\\
Policy & Causal/Non-state & \textbf{Memoryless, State} \\
Objective & \textbf{Flexible} & Quadratic \\
CCs & \textbf{Convex} & Nonconvex \\
Number of LMIs \cite{rapakouliasDiscreteTimeOptimalCovariance2023} & $N$ & $\mathbf{1}$ \\
Size of LMI(s) \cite{rapakouliasDiscreteTimeOptimalCovariance2023} & $\mathbf{n\times n}$ & $ n (N{+}1) \times n (N{+}1)$ \\
Complexity & High & \textbf{Low} \\
\midrule
Bottlenecks & Complexity & Numerical, CC nonconvexity\\
\bottomrule
\end{tabular}
\label{tab:comparison}
\end{table}

The numerical issues encountered due to using the covariance matrix as the dependent variable can be explained by the \textit{square root filter} literature. As early as the 1960s, engineers developing the Apollo Guidance Computer (AGC) were aware of the numerical instability of the standard Kalman filter \cite{grewalApplicationsKalmanFiltering2010}. Potter's equivalent derivation of the filter which uses the Cholesky factor of the covariance matrix as the dependent variable gave rise to the square root Kalman filter; achieving the same accuracy with ``half as many bits" of precision, it was adopted for the AGC \cite{grewalApplicationsKalmanFiltering2010}. Since then, square root methods have been developed to include process noise \cite{dyerExtensionSquarerootFiltering1969}, and alternative LDL factorization-based methods have been proposed \cite{thorntonGramSchmidtAlgorithmsCovariance1975}. Thornton and Bierman later document a numerical comparison of the standard versus the square root methods, demonstrating the divergence of the former for a spacecraft orbit determination scenario \cite{thorntonNumericalComparisonDiscrete1976}. See \cite{kaminskiDiscreteSquareRoot1971,tapleyStatisticalOrbitDetermination2004} for a general review and comparison of these methods.

Based on these observations and literature, we propose an alternative solution method for cc-CS, explicitly defining the covariance matrix Cholesky factor as the dependent variable. Due to the nonconvex representation of the propagation equation of the Cholesky factor, the problem is solved iteratively via sequential convex programming (SCP), utilizing the derivative of the QR decomposition as a key element. Nevertheless, the approach has the following benefits:
\begin{itemize}
	\item It achieves global optimality when the objective function is an expectation-of-quadratic form and CCs are not present. When CCs are present, the approach achieves local optimality.
	\item Compared to the block-matrix type methods \cite{okamotoOptimalCovarianceControl2018,okamotoOptimalStochasticVehicle2019}, it is more computationally efficient and shows better optimality. It is also memoryless, using only a single feedback gain matrix at each time step.
	\item Compared to the full-covariance type methods \cite{liuOptimalCovarianceSteering2025,pilipovskyComputationallyEfficientChance2024} which can fail for problems with small variance, it can achieve solutions with better numerical reliability. The choice of objective function is flexible. We remark that our method uses SCP, but such an iterative approach is nevertheless required for cc-CS using the approach in \cite{liuOptimalCovarianceSteering2025}.
\end{itemize}


We note that this work is not the first to propagate the square root of the covariance matrix in CS. In an earlier work by one of the authors \cite{oguriStochasticSequentialConvex2022}, this was attempted with some shortcomings: the variable that represents the square root grows proportional to the time step, yielding computational complexity similar to the block-matrix approaches. In addition, the feedback gains were retrieved inexactly via least-squares.

\section{Preliminaries}

\subsection{Notation and Preliminaries}
$\mathbb{L}^n$ ($\mathbb{L}_+^n$) denotes the space of $n \times n$ LT matrices (with positive diagonal entries). $\mathbb{S}^n_+$ denotes the space of $n \times n$ symmetric PD  matrices. The Cholesky factor of a matrix \( P \in \mathbb{S}_+^{n} \) is defined as the unique matrix \( S \in \mathbb{L}_+^n \) such that \( P = S S^\top \) \cite[Theorem 4.2.7]{golubMatrixComputations2013}; we write \(S = P^{1/2}\). $\mathrm{diag}(M)$ denotes the vector of diagonal elements of matrix \( M \), and $\mathrm{diag}(a,b, \cdots)$ denotes the diagonal matrix with $a,b, \cdots$ on the diagonals. $A \succ (\succeq) B$ denotes that $A - B $ is positive (semi)definite. \([N] := \{0, 1, \cdots, N\}\) for $N \in \N$.
We refer to the ``square root" of the covariance matrix \(P\) to refer to (not necessarily square) matrices \(L\) that satisfy \(P = L L^\top\). 
$\mathrm{tril}(\cdot)$ sets all but the LT part of a matrix to zero.
$\mathrm{vec}(\cdot)$ vectorizes a matrix.
$\mathrm{vectril}(\cdot)$ vectorizes the LT part of a matrix.
$\cov{\cdot}$ denotes the covariance matrix of a random variable.

\subsection{Problem Formulation}
We address the finite-horizon discrete-time CS problem for a linear time-varying system with additive Gaussian noise. The goal is to find a state feedback controller that steers the state mean and covariance from their initial values to desired terminal values while minimizing a cost function and obeying CCs. Mathematically, the problem is formulated as follows.
\begin{subequations}
\label{eq:problem_statement}
\begin{align}
	& \mathrlap{\minimize_{\bm{v}, \bm{K}} \quad J = \sum_{k=0}^{N-1} J_k (x_k, u_k) }\\
	& \suchthat &&x_{k+1} = A_k x_k + B_k u_k + G_k w_k, \quad k \in [N{-}1] \\
	& && u_k = v_k + K_k (x_k - \E{x_k}), \quad k \in [N{-}1] \label{eq:feedback_control_law}\\
	& &&x_0 \sim \mathcal{N}(\mu_{\mathrm{init}}, P_{\mathrm{init}}) \\
	& &&\mu_{\mathrm{fin}} = \E{x_N}\\
	& &&P_{\mathrm{fin}} \succeq \E{(x_N - \E{x_N})(x_N - \E{x_N})^\top} \\
	& &&\mathbb{P}(x_k \in \mathcal{X}) \geq 1 - \epsilon_x, \quad k \in [N] \label{eq:state_chance_constraint}\\ 
	& &&\mathbb{P}(u_k \in \mathcal{U}) \geq 1 - \epsilon_u, \quad k \in [N{-}1] \label{eq:control_chance_constraint}
\end{align}
\end{subequations}
where \( x_k \in \mathbb{R}^n \) is the state vector, \( u_k \in \mathbb{R}^m \) is the control input, and \( w_k \in \mathbb{R}^n \) is standard Gaussian process noise. \( \bm{v} := \{v_k\}_{k=0}^{N-1}\) and \(\bm{K} := \{K_k\}_{k=0}^{N-1}\) are respectively the nominal control input and the state feedback gain to be designed, with \( v_k \in \mathbb{R}^m \) and \( K_k \in \mathbb{R}^{m \times n} \). Note that due to the affine feedback control law, the state and the control input are Gaussian at each time step.
The matrices \( A_k \), \( B_k \), and \( G_k \) are of appropriate dimensions.

We consider two types of cost functions:
\begin{equation} \label{eq:objective_function}
	J_k (x_k, u_k) = \begin{cases}
	\E{x_k^\top Q_k x_k + u_k^\top R_k u_k} \\
	\mathcal{Q}_{\|W_{k}^x x_k\|_2}(p_J) +
	\mathcal{Q}_{\|W_{k}^u u_k\|_2}(p_J)
	\end{cases}
\end{equation}
The former is a standard \textbf{expectation-of-quadratic} (EoQ) cost with \( Q_k \succeq 0 \) and \( R_k \succ 0 \).
The latter is a \textbf{quantile-of-norm} (QoN) cost, where the quantile function is defined as:
\begin{equation}
	\mathcal{Q}_{T}(\gamma) = \inf \{ t > 0 : \mathbb{P}(T \leq t) \geq \gamma \}.
\end{equation}
$W_k^x \in \mathbb{R}^{w_x \times n}$ and $W_k^u \in \mathbb{R}^{w_u \times m}$ are weighting matrices for the state and control, respectively. Such a cost formulation is useful for risk sensitive scenarios where the worst-case scenario (up to a specified probability) is of more interest than the expected value, such as in finance \cite{rockafellarOptimizationConditionalValueatrisk2000}, statistical machine learning \cite{eastwoodProbableDomainGeneralization2022}, and spacecraft mission design \cite{oguriChanceConstrainedControlSafe2024a}.

The sets \( \mathcal{X} \) and \( \mathcal{U} \) represent state and control feasible sets, respectively, with associated violation probabilities \( \epsilon_x \) and \( \epsilon_u \). Consider CCs in either a linear or a norm constraint
\begin{align} 
	\mathbb{P}(\alpha_j ^\top x_k \leq \beta_j) \geq 1 - p_j \label{eq:linear_chance_constraint}\\
	\mathbb{P}(\| x_k \|_2 \leq \gamma) \geq 1 - p_\gamma \label{eq:norm_chance_constraint}
\end{align}
with \( \alpha_j \in \mathbb{R}^n \), \( \beta_j \in \mathbb{R} \), \(\gamma\in \mathbb{R}\), \( \gamma > 0 \) and violation probabilities  \( p_j \), \( p_\gamma \). The same forms apply to the input \( u_k \).

Let $P_k := \cov{x_k}$.
We make the following assumption, reasonable in practical applications and also made in \cite{liuOptimalCovarianceSteering2025}:
\begin{assumption} \label{assump:positive_definite_state}
	$P_k \succ 0,\quad \forall k \in [N].$
\end{assumption}

\subsection{QR Decomposition}

The QR decomposition is a key element of propagating the square root of the covariance matrix, \textit{while keeping the square root to be in \(\mathbb{L}^n\)}. Various algorithms have been proposed for the time update of the square root Kalman filter, but many can be expressed as an ``economy-size'' QR decomposition \cite{tapleyStatisticalOrbitDetermination2004}.

\begin{lemma}[Ch. 3.8 of \cite{ipsenNumericalMatrixAnalysis2009}]\label{lemma:existence_and_uniqueness_of_economy_size_qr_decomposition}
	Let \( M \in \mathbb{R}^{p \times q} \), \(p \geq q\).
	\begin{itemize}
		\item (Existence) There exist \(Q \in \mathbb{R}^{p \times q}\) and \(R \in \mathbb{R}^{q \times q}\) such that \(Q^\top Q = I_q\), \(R\) is upper triangular, and \(M = QR\).
		\item (Uniqueness) If \(M\) has full rank, i.e. \(\mathrm{rank}(M) = q\), then the matrices \(Q\) and \(R\) are unique up to a sign change of the diagonal entries of \(R\). If we require the diagonal entries of \( R \) to be positive, then the decomposition is unique.
	\end{itemize}
\end{lemma}
From here on, ``QR decomposition'' refers to the economy-size QR decomposition with positive diagonal entries\footnote{In MATLAB, this requires using $\mathrm{qr}(M, \text{``econ"})$. Then, to ensure uniqueness, if the function returns an \( R \) matrix with negative diagonal entries, multiply the corresponding columns of \( Q \) and rows of \( R \) by -1 to ensure positive diagonals.}.
Note that when \(M\) is full rank, then so is \(R\), hence \(R\) is invertible.

\subsection{QR Decomposition Derivative}
To handle square root covariance propagation in optimization, we use the derivative of the QR decomposition operation (i.e. the sensitivity of $Q$ and $R$ w.r.t. perturbations in $M$). $\dd R$ is written as a function of \(Q\), \( R\), and  $\dd M$ \cite{deleeuwDifferentiatingQRDecomposition2023}:
\begin{equation}
	\dd R(\dd M, Q, R) = Q^\top (\dd M) - \mathrm{tril}(Q^\top (\dd M) R^{-1}) R.
\end{equation}
Using the Taylor expansion, we can approximate the QR decomposition as a linear function of perturbations in \( M \), around some reference \( \bar{M} = \bar{Q} \bar{R} \):
\begin{equation}
	\mathrm{qr}(M) \approx \bar{R} + \dd R(M - \bar{M}, \bar{Q}, \bar{R}).
\end{equation}
where \( \mathrm{qr}(\cdot) \) denotes the operation of extracting the upper triangular factor \( R \) from the QR decomposition.
\( \dd Q \) can also be computed as a function of \( \dd M\) \cite{deleeuwDifferentiatingQRDecomposition2023}; it is not used for this work, so we omit it here.

\section{Deterministic Problem Formulation}
By applying the previous results, we can formulate \cref{eq:problem_statement} as a nonconvex deterministic optimization problem.
\subsection{Closed-Loop Square Root Covariance Propagation}
Since the state is controlled by the affine feedback control law, the closed-loop dynamics are given for $k \in [N{-}1]$ as
\begin{equation}
	x_{k+1} = A_k x_k + B_k (v_k + K_k (x_k - \E{x_k})) + G_k w_k.
\end{equation}
It is straightforward to derive the mean/covariance dynamics:
\begin{align} \label{eq:mean_dynamics}
	\mu_{k+1} &= A_k \mu_k + B_k v_k , \quad \mu_k := \E{x_k}\\
\label{eq:full_covariance_propagation}
	P_{k+1}& = (A_k + B_k K_k) P_k (A_k + B_k K_k)^\top + G_k G_k^\top.
\end{align}
Letting $S_k := P_k^{1/2}$, define the following matrix:
\begin{equation} \label{eq:X_definition}
	X_{k+1} := \begin{bmatrix} (A_k + B_k K_k) S_k & G_k \end{bmatrix}.
\end{equation}
\cref{lemma:rank_M_M_top_rank_M,lemma:positive_definite_iff_rank_n} are used to show \cref{lemma:rank_X_n_iff_P_succ_0}. They are standard results so we omit the  proofs.
\begin{lemma} \label{lemma:rank_M_M_top_rank_M}
	For any $M \in \mathbb{R}^{p \times q}$, $\rank(M M^\top) = \rank(M).$
\end{lemma}
\begin{lemma} \label{lemma:positive_definite_iff_rank_n}
	A symmetric positive semidefinite matrix $P$ of size $n \times n$ is strictly PD if and only if $\rank(P) = n$.
\end{lemma}
\begin{lemma} \label{lemma:rank_X_n_iff_P_succ_0}
	Let $P_{k+1}$, $X_{k+1}$ be given by \cref{eq:full_covariance_propagation}, \cref{eq:X_definition}. Then,
	\begin{equation}
		\rank(X_{k+1}) = n \Leftrightarrow P_{k+1} \succ 0.
	\end{equation}
\end{lemma}
\begin{proof}
	We use \cref{lemma:rank_M_M_top_rank_M,lemma:positive_definite_iff_rank_n} throughout the proof. \\
	($\Rightarrow$) 
	$\rank(X_{k+1}) {=} n \Rightarrow \rank(X_{k+1} X_{k+1}^\top) = n \Rightarrow P_{k+1} \succ 0$.\\
	($\Leftarrow$)
	Define $\widetilde{P}_{k+1} := (A_k + B_k K_k) P_k (A_k + B_k K_k)^\top$.
	\begin{align*}
	P_{k+1} \succ 0 &\Rightarrow \widetilde{P}_{k+1} \succ 0 \lor G_k G_k^\top \succ 0 \\
	&\Rightarrow \rank(\widetilde{P}_{k+1}) = n \lor \rank(G_k G_k^\top) = n \\
	&\Rightarrow \rank((A_k + B_k K_k) S_k) = n \lor \rank(G_k) = n \\
	&\Rightarrow \rank(X_{k+1}) = n.
	\end{align*}
	The last implication is since the column space of $X_{k+1}$ includes the column space of its submatrices.
\end{proof}

\begin{lemma}[Closed-Loop Square Root Covariance Propagation]
	\label{lemma:sqrt_propagation}
	The covariance propagation \cref{eq:full_covariance_propagation} is equivalently expressed in Cholesky factor form using $S_k = P_k^{1/2} \in \mathbb{L}_+^n$ as:
	\begin{equation} \label{eq:square_root_covariance_propagation_nonlinear_feedback_control}
		S_{k+1} = \mathrm{qr}(X_{k+1}^\top)^\top
	\end{equation}
	with $X_{k+1}$ defined as in \cref{eq:X_definition}.
\end{lemma}

\begin{proof}
The covariance propagation can be expressed as:
\begin{equation} \label{eq:covariance_propagation}
	S_{k+1} S_{k+1}^\top = X_{k+1} X_{k+1}^\top.
\end{equation}
Note that we want $S_{k+1} \in \mathbb{L}_+^n$, while $X_{k+1} \in \mathbb{R}^{n \times (n+m)}$.
For a matrix $Q_{k} \in \R^{(n+m) \times n}$ such that $Q_{k}^\top Q_k= I_n$,
\begin{subequations}
\begin{align}
     X_{k+1}^\top &= Q_{k} S_{k+1}^\top  \label{eq:Q_S_X_relation_1}\\
	\Rightarrow X_{k+1} X_{k+1}^\top &= S_{k+1} Q_{k}^\top Q_{k} S_{k+1}^\top = S_{k+1} S_{k+1}^\top.
\end{align}%
\label{eq:Q_S_X_relation}
\end{subequations}%
Since $S_{k+1} \in \mathbb{L}_+^n$, its transpose is upper triangular, so \cref{eq:Q_S_X_relation_1} is a QR decomposition of $X_{k+1}^\top$. 
Combining \cref{lemma:rank_X_n_iff_P_succ_0} and \cref{assump:positive_definite_state}, $X_{k+1}$ is full rank.
Then, from \cref{lemma:existence_and_uniqueness_of_economy_size_qr_decomposition}, there exists a unique set of matrices $(Q_{k}, S_{k+1})$ that satisfy the equation.
Thus, this QR decomposition is unique. 
\end{proof}
\cref{lemma:sqrt_propagation} ensures that $S_{k+1}$ obtained from \cref{eq:square_root_covariance_propagation_nonlinear_feedback_control} provides a unique matrix in $\mathbb{L}_n^+$ such that \cref{eq:full_covariance_propagation} is satisfied.

Since both $S_k$ and $K_k$ are unknowns, \cref{eq:square_root_covariance_propagation_nonlinear_feedback_control}, via \cref{eq:X_definition}, has a bilinear term \( K_k S_k \); for the sake of convexity, we perform a change of variables and define $L_k = K_k S_k \in \mathbb{R}^{m \times n}$. Then, $K_k$ can be uniquely recovered as $K_k = L_k S_k^{-1}$. The inverse exists since $S_k$ is the Cholesky factor of \( P_k \), which is PD by \cref{assump:positive_definite_state}. Note that $L_k L_k^\top = K_k S_k S_k^\top K_k^\top = K_k P_k K_k^\top = \cov{u_k}$. With this change of variables, we get
\begin{equation} \label{eq:square_root_covariance_propagation}
	S_{k+1} = \mathrm{qr}(X_{k+1}^\top)^\top, \ X_{k+1} = \begin{bmatrix} A_k S_k + B_k L_k & G_k \end{bmatrix}
\end{equation}
with its first-order approximation given by:
\begin{equation} \label{eq:square_root_covariance_propagation_linearized}
	\hspace{-1pt}S_{k+1} \approx \mathrm{qr}(\bar{X}_{k+1}^\top)^\top + \dd R( \Delta X_{k+1}^\top , \bar{Q}_{k+1}, \bar{R}_{k+1})^\top
\end{equation}
with \( \bar{X}_{k+1} := \begin{bmatrix} A_k \bar{S}_k + B_k \bar{L}_k & G_k \end{bmatrix} \), \(\bar{S}_k\), and \(\bar{L}_k\) being the reference values, and \(\bar{Q}_{k+1} \bar{R}_{k+1} = \bar{X}_{k+1}\) being the QR decomposition of \(\bar{X}_{k+1}\). 
Define $\Delta X_{k} := X_k - \bar{X}_k$.

\subsection{Objective Function}
After substituting \( u_k = K_k x_k = L_k S_k^{-1} x_k \), and performing some standard manipulations, the objective function can be expressed as deterministic convex functions of \( L_k \) and \( S_k \):

{
\setlength{\abovedisplayskip}{0pt} 
\setlength{\belowdisplayskip}{0pt} 
\small
\begin{equation} \label{eq:objective_deterministic}
	\hspace{-3pt}
	J = \begin{cases}
	\sum_{k=0}^{N-1} \mu_k^\top Q_k \mu_k + v_k^\top R_k v_k & \\
	\quad + \mathrm{trace} (Q_k S_k S_k^\top + R_k L_k L_k^\top) & \text{(EoQ)} \\
	\sum_{k=0}^{N-1} \norm{W_k^x x_k}_2 + \sqrt{\mathcal{Q}_{\chi^2_{w_x}}(p_J)} \ \norm{W_k^x S_k}_2 & \\
	\quad + \norm{W_k^u v_k}_2 + \sqrt{\mathcal{Q}_{\chi^2_{w_u}}(p_J)} \ \norm{W_k^u L_k}_2 & \text{(QoN)}.
	\end{cases}
\end{equation}
\setlength{\abovedisplayskip}{0pt} 
\setlength{\belowdisplayskip}{0pt} 
}

The QoN expression is a convex upper bound based on the triangle inequality \cite{oguriChanceConstrainedControlSafe2024a}; \( \mathcal{Q}_{\chi^2_d}(\cdot) \) is the inverse of the chi-squared cumulative distribution function with \( d \) degrees of freedom.

\begin{remark} \label{remark:convexifying_term}
	Although an approximation, the QoN cost is easily handled in the proposed method as a convex function of the variables (also by the block matrix methods \cite{ridderhofChanceConstrainedCovarianceControl2020,oguriChanceConstrainedControlSafe2024a}). The seminal work \cite{liuOptimalCovarianceSteering2025} relies on an assumption on the objective function for convexification, and we are not aware of a convex upper bound for this formulation. Workarounds such as linearizarion break the proof of lossless convexification in \cite{liuOptimalCovarianceSteering2025}. To retain the proof, an extra regularization term \(\eta \trace(Y_k)\) with small \( \eta > 0 \) can be added, as we do in \cite{kumagaiRobustCislunarLowThrust2025} ($Y_k := \cov{u_k}$).
\end{remark}

\subsection{Constraints}
The square root representation allows all other constraints to be rewritten as convex functions of the variables.

The structure of \(S_k\) is imposed by
\begin{equation} \label{eq:S_structure_constraint}
	S_k \in \mathbb{L}^n, \quad \mathrm{diag}(S_k) \geq 0, \quad k \in [N].
\end{equation}
The initial distributional constraint is imposed by
\begin{equation} \label{eq:initial_distributional_constraint}
	\mu_0 = \mu_{\mathrm{init}}, \quad S_0 = P_{\mathrm{init}}^{1/2}.
\end{equation}
The terminal distributional constraint is imposed by \cite{okamotoOptimalCovarianceControl2018}
\begin{equation} \label{eq:terminal_distributional_constraint}
	\mu_N = \mu_{\mathrm{fin}}, \quad \|P_{\mathrm{fin}}^{-1/2} S_N\|_2 - 1 \leq 0.
\end{equation}
The affine CC in \cref{eq:linear_chance_constraint} can be exactly expressed as \cite{okamotoOptimalCovarianceControl2018}:
\begin{equation} \label{eq:linear_chance_constraint_deterministic}
	\alpha_j^\top \mu_k + \mathcal{Q}_{\mathcal{N}}(1 - p_j) \|S_k^\top \alpha_j\|_2 - \beta_j \leq 0, \quad k \in [N]
\end{equation}
where \( \mathcal{Q}_{\mathcal{N}}(\cdot) \) is the inverse of the standard normal cumulative distribution function.
For the norm CC, similar to the QoN cost, the CC can be conservatively imposed as \cite{oguriChanceConstrainedControlSafe2024a}:
\begin{equation} \label{eq:norm_chance_constraint_deterministic}
	\norm{\mu_k}_2 + \sqrt{Q_{\chi^2_n}(1 - p_\gamma)} \ \norm{S_k}_2 - \gamma \leq 0, \quad k \in [N]
\end{equation}
To impose control CCs, replace \( \mu_k \) with \( v_k \) and \( S_k \) with \( L_k \).

To summarize, \cref{eq:problem_statement} is converted into a nonconvex optimization problem as follows:
\begin{equation} \label{eq:nonconvex_problem_statement}
	\minimize_{\bm{v}, \bm{L}, \bm{\mu}, \bm{S}} \cref{eq:objective_deterministic} \quad \text{s.t.} 
	\quad 
	\text{\cref{eq:S_structure_constraint,eq:mean_dynamics,eq:square_root_covariance_propagation,eq:initial_distributional_constraint,eq:terminal_distributional_constraint,eq:linear_chance_constraint_deterministic,eq:norm_chance_constraint_deterministic}}
\end{equation}
where the only source of nonconvexity is the square root covariance dynamics \cref{eq:square_root_covariance_propagation}. For notational simplicity, we define \(\bm{L} := \{L_k\}_{k=0}^{N-1}\), \(\bm{\mu} := \{\mu_k\}_{k=0}^{N}\), and \(\bm{S} := \{S_k\}_{k=0}^{N}\).
\section{Algorithm Overview}
In order to take advantage of the convex representation of the objective and constraints, we use the SCP algorithm \texttt{SCvx*} \cite{oguriSuccessiveConvexificationFeasibility2023a} to solve \cref{eq:nonconvex_problem_statement}; this ensures that the converged solution obeys nonlinear dynamics \cref{eq:square_root_covariance_propagation_nonlinear_feedback_control}and is locally optimal. In general, any SCP algorithm can be used to ensure the satisfaction of \cref{eq:square_root_covariance_propagation_nonlinear_feedback_control}, but local optimality must be met in order to apply the results of \cref{th:global_optimality,th:local_optimality}. At each iteration, the linearized dynamics \cref{eq:square_root_covariance_propagation_linearized} are imposed with slack variables,\par
{
\setlength{\abovedisplayskip}{0pt} 
\setlength{\belowdisplayskip}{0pt} 
\small
\begin{equation} \label{eq:linearized_dynamics}
\mathrm{vectril} \{ S_{k+1} - \mathrm{qr}(\bar{X}_{k+1}^\top)^\top {-} \dd R(\Delta X_{k+1}^\top, \bar{Q}_{k+1}, \bar{R}_{k+1})^\top \} = \xi_k
\end{equation}
}%
where \( \xi_k \in \mathbb{R}^{n(n+1)/2}\) is penalized in the objective function to promote constraint satisfaction. 
The reference QR matrices are computed at initialization and updated when a step is accepted:
\begin{equation} \label{eq:reference_update}
	\{\bar{Q}_{k+1}, \bar{R}_{k+1}\} \gets \mathrm{qr}(\begin{bmatrix} A_k \bar{S}_k + B_k \bar{L}_k & G_k \end{bmatrix}).
\end{equation}
To avoid artificial infeasibility/unboundedness \cite{maoSuccessiveConvexificationNonConvex2016} due to linearization, we impose a trust region constraint:
\begin{align} \label{eq:trust_region_constraint}
	\|\mathrm{vec}(D_X\cdot \Delta X_k)\|_\infty &\leq r, \quad k \in [N{-}1]
\end{align}
where \( r > 0 \) is the trust region radius, and \( D_X \) is a scaling matrix.
Trust regions on \( \mu_k \) and \( v_k \) are not necessary, since they only appear in the convex constraints.

The objective function is augmented with a penalty term using the augmented Lagrangian method. The penalty function for the constraint violations is defined as:
\begin{equation} \label{eq:penalty_function}
	P(\bm{\xi}; w, \bm{\lambda}) = \sum_{k=0}^{N-1} \lambda_k^\top \xi_k + \frac{w}{2} \|\xi_k\|_2^2.
\end{equation}
$\bm{\lambda} := \{\lambda_k\}_{k=0}^{N-1}$ with \( \lambda_k \in \mathbb{R}^{n(n+1)/2} \) are Lagrange multipliers; \( w > 0\) is the penalty weight; \(\bm{\xi} := \{\xi_k\}_{k=0}^{N-1}\) are the slack variables from \cref{eq:linearized_dynamics}. Define \(\bm{z} := \{\bm{v}, \bm{L}, \bm{\mu}, \bm{S} \} \). The augmented objective function is then:
\begin{equation} \label{eq:objective_augmented}
	\mathcal{L}(\bm{z}, \bm{\xi}; w, \bm{\lambda}) := 
	J(\bm{z}) + P(\bm{\xi}; w, \bm{\lambda}).
\end{equation}

The convex subproblem is given by:
\begin{equation} \label{eq:subproblem}
	\minimize_{\bm{v}, \bm{L}, \bm{\mu}, \bm{S}, \bm{\xi}} \cref{eq:objective_augmented} \quad \text{s.t.} \quad 
	\text{\cref{eq:S_structure_constraint,eq:mean_dynamics,eq:linearized_dynamics,eq:initial_distributional_constraint,eq:terminal_distributional_constraint,eq:linear_chance_constraint_deterministic,eq:norm_chance_constraint_deterministic,eq:trust_region_constraint}}.
\end{equation}
Algorithm \ref{alg:sqrt_qr_covariance_steering} shows the solution process. Every iteration, the subproblem \cref{eq:subproblem} is solved. Then, the following quantities determine the step acceptance and parameter updates:

{
\begin{align}
	\widetilde{\bm{\xi}}(\bm{z}) & := \begin{bmatrix}
		\mathrm{vectril}(S_1 - (\mathrm{qr}(X_1^{\top}))^\top) \\
		\vdots \\
		\mathrm{vectril}(S_N - (\mathrm{qr}(X_N^{\top}))^\top)
	\end{bmatrix} \\
	\mathcal{J}(\bm{z}, w, \bm{\lambda}) &:= J(\bm{z}) + P({\widetilde{\bm{\xi}}}(\bm{z}); w, \bm{\lambda}) \\
	\Delta J^{(i)} &:= \mathcal{J}(\bar{\bm{z}}; w^{(i)}, \bm{\lambda}^{(i)}) - \mathcal{J}(\bm{z}^*; w^{(i)}, \bm{\lambda}^{(i)}) \label{eq:delta_J} \\
	\Delta L^{(i)} &:= \mathcal{J}(\bar{\bm{z}}; w^{(i)}, \bm{\lambda}^{(i)}) - \mathcal{L}(\bm{z}^*, \bm{\xi}^*; w^{(i)}, \bm{\lambda}^{(i)}) \label{eq:delta_L}\\
	\chi^{(i)} &:= \|\widetilde{\bm{\xi}}(\bm{z})\|_2 \label{eq:chi}
\end{align}
}%
where the superscript $(i)$ denotes the SCP iteration number, and $^*$ denotes the solution obtained at the current iteration. Here, $\widetilde{\bm{\xi}}$ is the vector of infeasibility of the nonlinear square root covariance dynamics, $\mathcal{J}$ is the augmented cost penalizing the nonlinear infeasibility,  $\Delta J^{(i)}$ is the actual cost reduction, $\Delta L^{(i)}$ is the predicted cost reduction from the linearized subproblem, and $\chi^{(i)}$ measures the norm of nonlinear infeasibility. The step acceptance ratio $\rho^{(i)} = \Delta J^{(i)} / \Delta L^{(i)}$ is computed, where $\Delta L^{(i)} \geq 0$ is guaranteed by the convexity of the subproblem \cite{oguriSuccessiveConvexificationFeasibility2023a}. If $\Delta L = 0$, one can set $\rho^{(i)} = 1$. 
Readers are referred to \cite{oguriSuccessiveConvexificationFeasibility2023a} for more details of $\texttt{SCvx*}$, including its proof of convergence to a local minimum.

\begin{algorithm}[t]
	\caption{Square Root-Factorized CS via \texttt{SCvx*}}
	\label{alg:sqrt_qr_covariance_steering}
	\begin{algorithmic}[1]
	\Require Initial guess $(\bar{\bm{L}}, \bar{\bm{S}})$, SCP parameters ($\epsilon_{\mathrm{feas}}$, $\epsilon_{\mathrm{opt}}$, $\rho_0$, $\rho_1$, $\rho_2$, $\alpha_1$, $\alpha_2$, $\beta$, $\gamma$, $r_{\mathrm{min}}$, $r_{\mathrm{max}}$, $r_{\mathrm{init}}$, $w_{\mathrm{init}}$)
	\State $i {\gets} 1$, $r^{(1)} {\gets} r_{\mathrm{init}}$, $w^{(1)} {\gets} w_{\mathrm{init}}$, $\delta^{(1)} {\gets} \infty$, $\bm{\lambda}^{(1)} {\gets} \bm{0}$
	\State Compute reference QR matrices using \cref{eq:reference_update}
	\While{$|\Delta J^{(i)}| > \epsilon_{\mathrm{opt}}$ or $\chi^{(i)} > \epsilon_{\mathrm{feas}}$}
		\State $(\bm{v}^*, \bm{L}^*, \bm{\mu}^*, \bm{S}^*, \bm{\xi}^*) \gets \arg\min$ \cref{eq:subproblem}
		\State Compute $\Delta J^{(i)}$, $\Delta L^{(i)}$, and $\chi^{(i)}$ using \cref{eq:delta_J,eq:delta_L,eq:chi}
		\State Compute step acceptance ratio: $\rho^{(i)} \gets \Delta J^{(i)} / \Delta L^{(i)}$
		\If{$\rho^{(i)} \geq \rho_0$} \Comment{Accept step}
			\State $(\bar{\bm{v}}, \bar{\bm{L}}, \bar{\bm{\mu}}, \bar{\bm{S}}) \gets (\bm{v}^*, \bm{L}^*, \bm{\mu}^*, \bm{S}^*)$
			\State Update reference QR matrices using \cref{eq:reference_update}
			\If{$|\Delta J^{(i)}| < \delta^{(i)}$}
				\State Set $\bm{\lambda}^{(i+1)}$, $\delta^{(i+1)}$, $w^{(i+1)}$ with \cite[Eq. (3)]{oguriSuccessiveConvexificationFeasibility2023a}
			\Else
			\State $(\bm{\lambda}^{(i+1)}, \delta^{(i+1)}, w^{(i+1)}) \gets (\bm{\lambda}^{(i)}, \delta^{(i)}, w^{(i)})$
			\EndIf
		\Else
			\State $(\bm{\lambda}^{(i+1)}, \delta^{(i+1)}, w^{(i+1)}) \gets (\bm{\lambda}^{(i)}, \delta^{(i)}, w^{(i)})$
		\EndIf
		\State Update $r^{(i+1)}$ with \cite[Eq. (17)]{oguriSuccessiveConvexificationFeasibility2023a}
		\State $i \gets i + 1$
	\EndWhile
	\State Recover feedback gains: $K_k^* \gets L_k^* (S_k^*)^{-1}$ for $k \in [N{-}1]$
	\State \Return $\bm{v}^*, \bm{K}^*, \bm{\mu}^*, \bm{S}^*$
	\end{algorithmic}
	\end{algorithm}

\section{Optimality of the Converged Solution}

In this section, we connect the optimality of the solution obtained from \cref{alg:sqrt_qr_covariance_steering} (in the $\mathbb{L}_+^n$ coordinates) to that obtained from solving it defining the covariance matrix as variables, as in \cite{liuOptimalCovarianceSteering2025,pilipovskyComputationallyEfficientChance2024} (in the $\mathbb{S}_+^n$ coordinates).

 Consider two constrained optimization problems on manifolds $\mathcal{M}_A$ and $\mathcal{M}_B$ with variables $x \in \mathcal{M}_A$ and $y \in \mathcal{M}_B$:%
\begin{equation*}
\text{(A)} \ \minimize_{x \in \mathcal{X}} \ f(x) \qquad
\text{(B)} \ \minimize_{y \in \mathcal{Y}} \ f(T(y))
\end{equation*}
where $T: \mathcal{M}_B \rightarrow \mathcal{M}_A$ is a diffeomorphism such that $T(\mathcal{M}_B) = \mathcal{M}_A$, with equality constraints $g_i, i \in \mathcal{I}$ and inequality constraints $h_j, j \in \mathcal{J}$:
\begin{align}
	\mathcal{X} &:= \{ x \in \mathcal{M}_A \ |\  g_i(x) = 0, i \in \mathcal{I}; \ h_j(x) \leq 0 \in \mathcal{J} \} \nonumber \\ 
	\mathcal{Y} &:= \{ y \in \mathcal{M}_B \ |\  g_i(T(y)) = 0, i \in \mathcal{I}; \ h_j(T(y)) \leq 0\in \mathcal{J} \} \nonumber 
\end{align}
From the definition, $\mathcal{X} = T(\mathcal{Y})$.

\begin{proposition}[Invariance of Local Optimality via a Diffeomorphism]
\label{prop:global_optimality_diffeomorphism}
Let $y^*$ be a local minimum of (B). Then $x^* = T(y^*)$ is a local minimum of (A).
\end{proposition}

\begin{proof}
Since $y^*$ is a local minimum of (B), there exists a neighborhood $\mathcal{N}$ around $y^*$ such that
\begin{equation}
\label{eq:local_minimum}
f(T(y)) \geq f(T(y^*)) \quad \forall y \in \mathcal{S}_B
\end{equation}
where $\mathcal{S}_B := N \cap \mathcal{Y}$ \cite[Chapter 12]{nocedalNumericalOptimization2006}.

Since $T$ is a diffeomorphism, it is in particular a homeomorphism. Therefore, $T$ maps the open neighborhood $\mathcal{N}$ to an open set $T(\mathcal{N})$ which contains $T(y^*) = x^*$ \cite[Chapter 12]{munkresTopology2014}.
Defining $\mathcal{S}_A := T(N) \cap \mathcal{X}$,
\begin{equation}
\mathcal{S}_A = T(\mathcal{N}) \cap T(\mathcal{Y}) 
= T(\mathcal{N} \cap \mathcal{Y}) 
= T(\mathcal{S}_B)
\end{equation}
where the second equality holds since $T$ is in particular a bijection \cite[Theorem 12.4]{hammackBookProof2018}.
Therefore, for any $x \in \mathcal{S}_A$, there exists a unique pre-image $y = T^{-1}(x) \in \mathcal{S}_B$. From \cref{eq:local_minimum},
\begin{equation}
f(T(y)) \geq f(T(y^*)),\ \forall y \in \mathcal{S}_B \nonumber
\Rightarrow f(x) \geq f(x^*),\ \forall x \in \mathcal{S}_A. \label{eq:local_optimality_x}
\end{equation}
Hence, \(x^*\) is a local minimum of (A).
\end{proof}

\begin{remark}
	Another way to show \cref{prop:global_optimality_diffeomorphism} is to show that the KKT conditions of (A) and (B) are equivalent. Readers are referred to Proposition 1 of \cite{fatkhullinGlobalSolutionsNonConvex2025}. The assumptions used in the proposition, 1. weak convexity of the objective and constraints and 2. differentiability of the coordinate transformation, are satisfied by the problem statement in \cref{eq:nonconvex_problem_statement}.
\end{remark}

Next, we apply this result to our specific problem. The following proposition is used in the proof.
\begin{proposition}[Proposition 2 of \cite{linRiemannianGeometrySymmetric2019}]
	\label{prop:cholesky_diffeomorphism}
	The Cholesky map \(\mathcal{L} : \mathbb{S}^n_+ \rightarrow \mathbb{L}_+^n\) is a diffeomorphism.
\end{proposition}

\begin{theorem} \label{th:local_optimality}
	\cref{alg:sqrt_qr_covariance_steering} converges to a locally optimal solution of the cc-CS problem. The local optimality is invariant to whether the problem is represented with the full covariance matrix or its Cholesky factor. 
\end{theorem}

\begin{proof}
	Let (A) from \cref{prop:global_optimality_diffeomorphism} be the determinisic formulation method in \cite{liuOptimalCovarianceSteering2025}. In this formulation, variable $x \in \mathcal{M}_A = \mathbb{S}_+^n \times \cdots \times \mathbb{S}_+^n \times \R^{n\times m} \times \cdots \times \R^{n \times m}$ where the first $N+1$ spaces are the PD matrices $P_k$ and the next $N$ spaces are the feedback gains $K_k$. 
	Let (B) from \cref{prop:global_optimality_diffeomorphism} be \cref{eq:nonconvex_problem_statement} before the change of variables $L_k = K_k S_k$. In this formulation, variable $y \in \mathcal{M}_B = \mathbb{L}_+^n \times \cdots \times \mathbb{L}_+^n \times \R^{n\times m} \times \cdots \times \R^{n \times m}$, where the first $N+1$ spaces are $S_k$ and the next $N$ spaces are the feedback gains $K_k$. (We can ignore the mean variables $\bm{v}, \bm{\mu}$ as they appear exactly in the same way in either problem.)
	(B) is solved via \cref{alg:sqrt_qr_covariance_steering}, and the diffeomorphism $T: \mathcal{M}_B \rightarrow \mathcal{M}_A$ maps each of the matrices $\mathbb{S}_+^n$ to $\mathbb{L}_+^n$ via the inverse of the Cholesky map \(\mathcal{L} : \mathbb{S}^n_+ \rightarrow \mathbb{L}_+^n\) and does not modify the feedback gains $K_k$. The inverse of a diffeomorphism is also a diffeomorphism.
	Then, the proof follows from the local optimality of the \texttt{SCvx*} algorithm \cite[Theorem 3]{oguriSuccessiveConvexificationFeasibility2023a} and \cref{prop:cholesky_diffeomorphism}. 
\end{proof}

\begin{theorem} \label{th:global_optimality}
	When the CCs are not present and the objective takes the EoQ form of \cref{eq:objective_function}, the solution obtained by \cref{alg:sqrt_qr_covariance_steering} is globally optimal.
\end{theorem}
\begin{proof}
	The unconstrained CS problem with EoQ cost is a convex optimization problem \cite{liuOptimalCovarianceSteering2025}. Since a local optimum of a convex problem is also a global optimum, from \cref{th:local_optimality}, the solution obtained from \cref{alg:sqrt_qr_covariance_steering} is globally optimal.
\end{proof}
\section{Numerical Examples}
For all experiments, we use MATLAB on a MacBook Pro with an M4 Max processor with 24GB of RAM, and  YALMIP \cite{lofbergYALMIPToolboxModeling2004} and MOSEK \cite{mosekapsMOSEKOptimizationToolbox2025} to solve SDPs. 
Code is available at {\small \url{https://github.com/naoya-kumagai/sqrt-cs-release/}}.
For the proposed method, the initial guess \( \bar{\bm{S}} \) is obtained by the covariance matrix interpolation technique from \cite{linRiemannianGeometrySymmetric2019}. The initial guess \( \bar{\bm{L}} \) is set to zero.


\subsection{Optimality and Scalability for the Unconstrained Case} \label{sec:comparison_unconstrained}
Here, we compare the horizon-length scalability and solution optimality of our method against Liu et al. \cite{liuOptimalCovarianceSteering2025} and Okamoto and Tsiotras \cite{okamotoOptimalStochasticVehicle2019}. We consider a three-dimensional double integrator with horizon of \(T = 3\) seconds, and vary the horizon length from 10 to 160 (so the step size is \(\Delta t = T / N\)). 
The time-invariant dynamics are given by
\begin{equation} \label{eq:double_integrator_equations}
	A = \smqty[
		I_3 & \Delta t I_3 \\
		0_{3 \times 3} & I_3
	],\ 
	B = \smqty[
		0.5 \Delta t^2 I_3 \\
		\Delta t I_3
	],\ 
	G = \sqrt{q \Delta t} I_6
\end{equation}
with \( q = 0.05 \) being the process noise spectral density.
The initial and terminal state covariance are \(I_6\) and \(0.5 I_6\), respectively. We use an EoQ cost with \(Q = 0.1 I_6\) and \(R = I_3\).
For our method, we set the convergence and feasibility tolerance to \( \epsilon_{\mathrm{opt}} = \epsilon_{\mathrm{feas}} = 10^{-4}\) and scaling matrix \(D_X\) to identity. To time the solutions, we perform 5 trials for each case. The results are shown in \cref{fig:comparison_of_methods}.

For this unconstrained problem, \cite{liuOptimalCovarianceSteering2025} offers fast and optimal solutions. The computational complexity of \cite{okamotoOptimalStochasticVehicle2019} grows quickly, which matches the observations in \cite{rapakouliasDiscreteTimeOptimalCovariance2023}. Due to its specific feedback policy, it also compromises optimality. In comparison, our approach has moderate computational complexity, with rate of growth similar to \cite{liuOptimalCovarianceSteering2025}. The optimality empirically validates \cref{th:global_optimality}. The original cc-CS method \cite{okamotoOptimalCovarianceControl2018} using state history feedback showed much more rapid growth in computational cost and is not included in this comparison.

\begin{figure}[t]
	\centering
	\begin{subfigure}{0.4\linewidth}
		\centering
		\includegraphics[width=\textwidth]{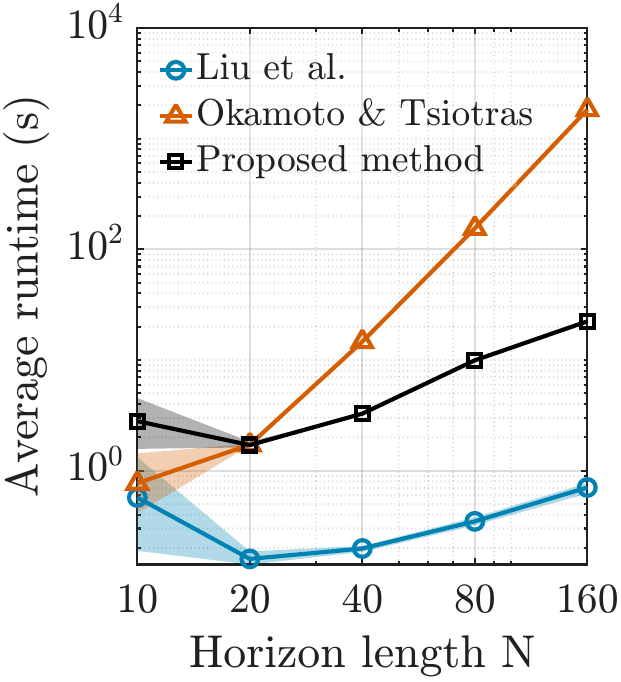}	
	\end{subfigure}%
	\begin{subfigure}{0.4\linewidth}
		\centering
		\includegraphics[width=\textwidth]{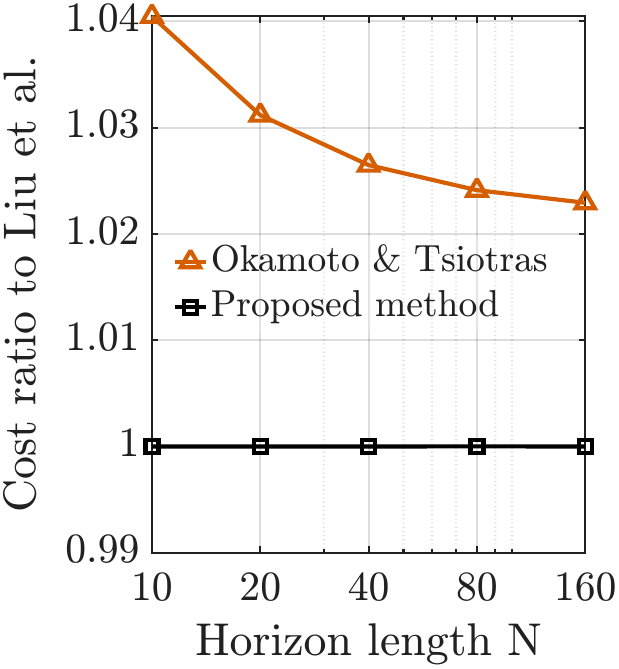}	
	\end{subfigure}
	\caption{Comparisons for the unconstrained case. (\textbf{Left}) Comparison of the average (in solid lines) and the min-max range (in shaded regions) of the solution times between \cite{liuOptimalCovarianceSteering2025,okamotoOptimalStochasticVehicle2019}, and the proposed method (\textbf{Right}) Comparison of the costs for \cite{okamotoOptimalStochasticVehicle2019} and the proposed method, relative to \cite{liuOptimalCovarianceSteering2025}}
	\label{fig:comparison_of_methods}
\end{figure}


\subsection{Optimality and Scalability for the Constrained Case}
Next, we perform a comparison of the same methods for path planning in an environment with a circular obstacle. We conservatively approximate the obstacle as a time-varying hyperplane \cite{dicairanoModelPredictiveControl2012}.
These hyperplanes are obtained by solving the deterministic version of the problem to obtain position histories for the linearization reference.
The setting is a 2D double integrator with $q = 0.005$ in \cref{eq:double_integrator_equations}. Between the boundary positions, there is a circular obstacle with its center at $[5, 0]$ and radius $1.2$. We define the nonconvex state CC: $\P{\norm{p - [5, 0]}_2 \geq 1.2} \geq 1 - 0.005$
for each node, where $p = [x, y]^\top$ is the position vector.
The control is chance-constrained as $\P{u \leq 0.15} \geq 1 - 0.005$, $\P{u \geq -0.15} \geq 1 - 0.005$
for each control component and node. The other parameters are $T = 30$, $ \mu_{\mathrm{init}} = [0, 0, 0, 0]^\top$, $ \mu_{\mathrm{fin}} = [10, 0, 0, 0]^\top$, $P_{\mathrm{init}} =\diag{0.1, 0.1, 0.01, 0.01}$, $P_{\mathrm{fin}} = \diag{0.04, 0.04, 0.01, 0.01}$.

When using Liu's formulation \cite{liuOptimalCovarianceSteering2025}, affine CCs are nonconvex in the variables due to the square root term that appears. Two methods exist for handling this nonconvexity: 1) providing a constant reference covariance matrix for all nodes \cite{rapakouliasDiscreteTimeOptimalCovariance2023}, and 2) rewriting the constraint as a difference-of-convex form, solved using the convex-concave procedure (CCP)\cite{pilipovskyComputationallyEfficientChance2024}. We use the latter method, as the former method conservatively approximates the available solution space; this resulted in \textit{artificial infeasibility}. A naive CCP also encounters this issue. Hence, we apply the Penalty CCP method \cite{lippVariationsExtensionConvex2016} which introduces positive slack variables and augments a penalty term in the objective function. The implementation of the Penalty CCP method follows the algorithm of \cite{lippVariationsExtensionConvex2016}. The iterative procedure ends with similar criteria as the proposed method: 1) the nonconvex CC violation and all slack varialbes are smaller than $\epsilon_{\mathrm{feas}}$ and 2) the difference of the augmented objective function with the previous step is smaller than $\epsilon_{\mathrm{opt}}$. The slack variables are multiplied a factor $\tau$ and added to the objective function. $\tau$ is initially 100 and increased by a factor of 2 at each iteration. For this formulation and the proposed method, we set $\epsilon_{\mathrm{opt}} = 10^{-2}$, $\epsilon_{\mathrm{feas}} = 10^{-4}$. 

\begin{figure}[t]
	\centering
	\begin{subfigure}{0.4\linewidth}
		\centering
		\includegraphics[width=\textwidth]{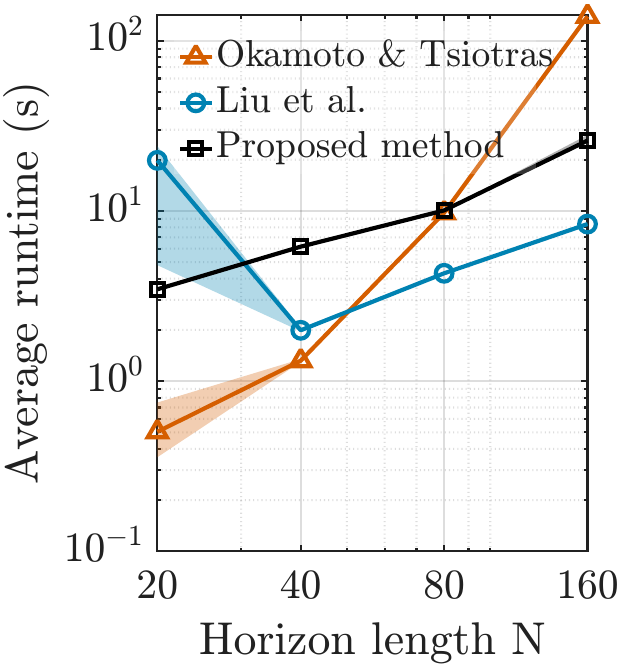}	
	\end{subfigure}%
	\begin{subfigure}{0.4\linewidth}
		\centering
		\includegraphics[width=\textwidth]{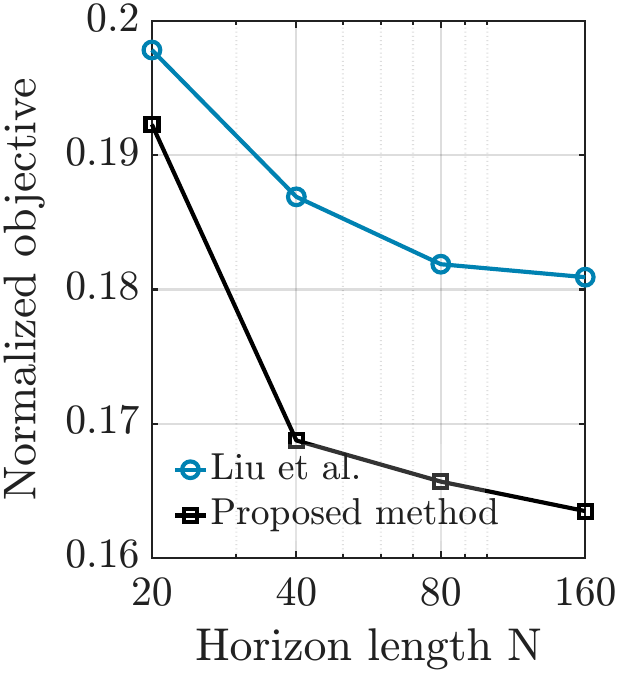}	
	\end{subfigure}%
	\caption{Comparisons for the obstacle-avoidance case. (\textbf{Left}) Average (in solid lines) and the min-max range (in shaded regions) of the solution times. (\textbf{Right}) Objective function divided by the horizon length $N$}
	\label{fig:comparison_of_methods_obstacle}
\end{figure} 

\begin{figure}[t]
	\centering
	\begin{subfigure}{0.5\linewidth}
		\centering
		\includegraphics[width=\textwidth]{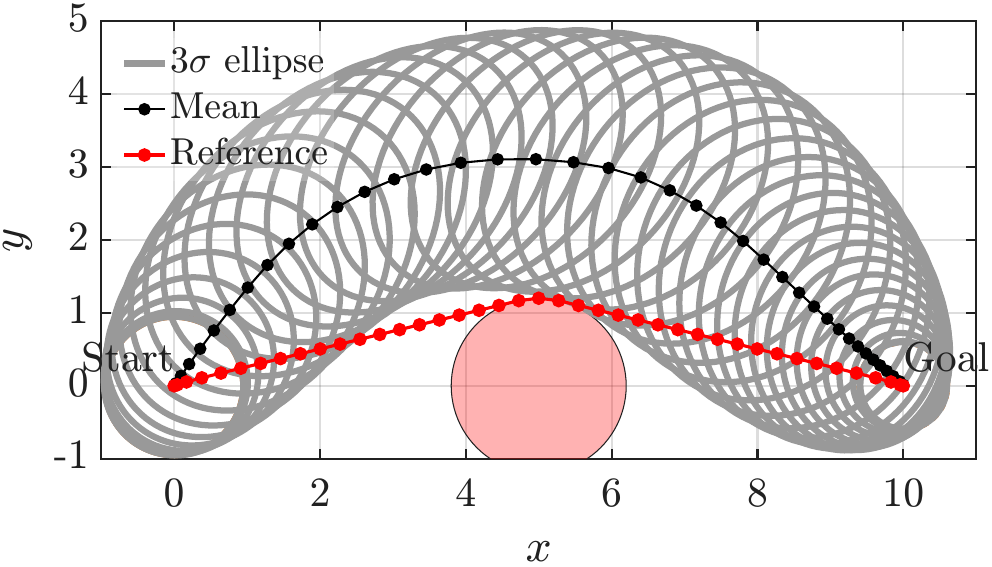}
	\end{subfigure}%
	\begin{subfigure}{0.5\linewidth}
		\centering
		\includegraphics[width=\textwidth]{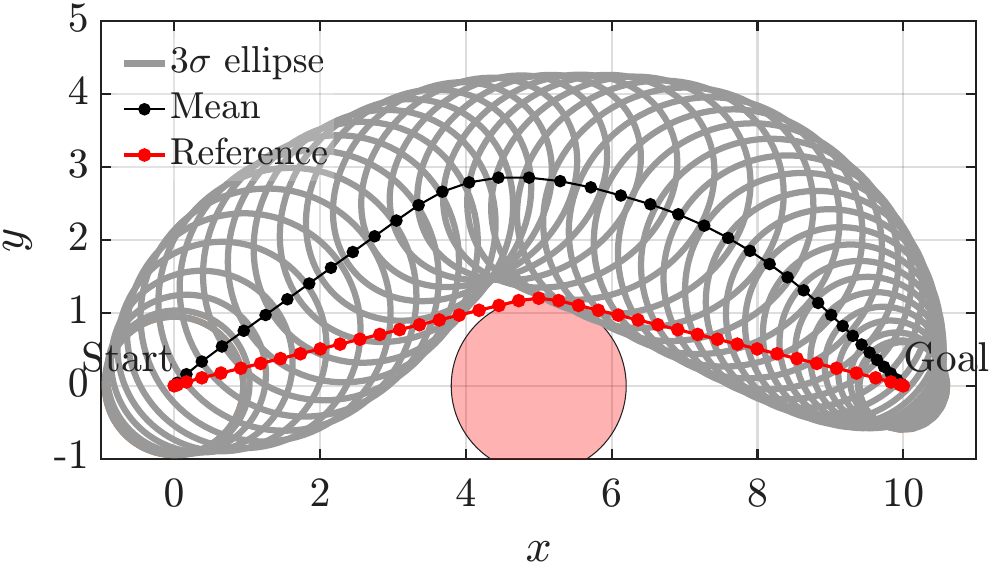}
	\end{subfigure}
	\caption{Obstacle avoidance with Liu's formulation with Penalty CCP (left) and proposed method (right) for the case with $N=40$}
	\label{fig:planning_with_obstacles_comparison}
\end{figure}

The comparison results are shown in \cref{fig:comparison_of_methods_obstacle}. 
For this constrained problem, the proposed method and CCP-based method have similar computational complexity. The proposed method achieves lower cost. \cite{okamotoOptimalStochasticVehicle2019} reported infeasibility for all cases; since MOSEK's interior-point solver required a typical number of iterations before terminating, we include it as a reference.
\cref{fig:planning_with_obstacles_comparison} compares the trajectory and covariance ellipses for the case with $N=40$. The two methods converge to different solutions, with a difference in how large the position covariance is allowed to grow.

\begin{remark}
	The proposed method solves cc-CS with approximately the same computational effort as unconstrained CS.
\end{remark}

\subsection{Numerical Reliability}
Here, we demonstrate the numerical reliability of the proposed method when the variance of some state components are small. We consider a problem similar to \cite{oguriChanceConstrainedControlSafe2024a} which uses CS for a spacecraft rendezvous setting. The problem uses a QoN objective function with $W_k^x = 0$ and $W_k^u = I$ for all $k$ \cref{eq:objective_function} which is common for optimizing the quantile of the spacecraft's fuel consumption. We emphasize that the parameters of uncertainty are rather realistic for a spacecraft rendezvous setting, and not artificial: initial mean position (velocity) $[-3.0, 0.126, 0]^\top$ km ($0_{3 \times 1}$ km/s); initial position (velocity) standard deviation $100$ m ($1.0$ m/s); terminal mean position (velocity) $[0.0, 0.05, 0]^\top$ km ($0_{3 \times 1}$ km/s); terminal position(velocity) standard deviation $10.0$ m ($0.1$ m/s); stochastic acceleration $1.0$ mm/s$^{3/2}$; horizon length $N = 14$; time step $30$ s; chief orbit radius $7228$ km; gravitational parameter $3.986 \times 10^{5}$ km$^3$/s$^2$.

The continuous-time Clohessy-Wiltshire-Hill equations \cite{schaubAnalyticalMechanicsSpace2003} are discretized using zero-order hold control, as described in \cite{oguriChanceConstrainedControlSafe2024a}. 
We use $D_X = \mathrm{diag}(10, 10, 10, 10^3, 10^3, 10^3)$ for the trust region scaling, chosen to make the magnitude of $D_X S_k$ close to unity. We set \( \epsilon_{\mathrm{opt}} = \epsilon_{\mathrm{feas}} = 10^{-5}\).
The proposed method solves the problem in 42 iterations in 5.0 seconds (average of 10 runs). The results are shown in \cref{fig:example_cwh_trajectory}, along with the results of 1000 Monte Carlo simulations. We can observe that the terminal covariance constraint is satisfied with equality, and the Monte Carlo samples are within the 3-sigma bounds.

\begin{figure}[t]
	\centering
	\includegraphics[width=\linewidth]{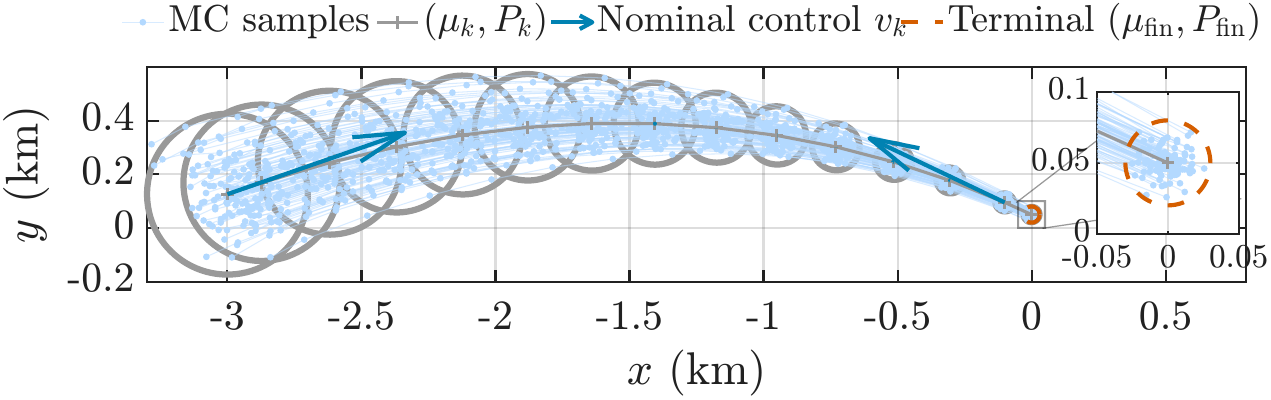}
	\caption{Projection of trajectory onto the $x$-$y$ plane for the spacecraft rendezvous problem, with $3\sigma$ covariance ellipses around the mean trajectory}
	\label{fig:example_cwh_trajectory}
\end{figure}
\begin{figure}[t]
	\centering
	\includegraphics[width=0.75\linewidth]{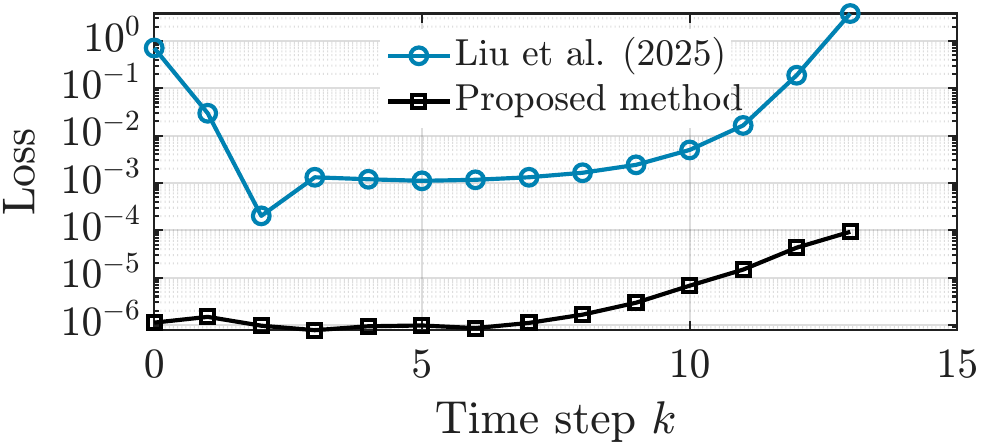}
	\caption{Covariance propagation loss for the spacecraft rendezvous problem}
	\label{fig:example_cwh_covariance_propagation_loss}
\end{figure}

For the approach of \cite{liuOptimalCovarianceSteering2025}, we use $\eta = 10^{-3}$ for the convexifying term (see \cref{remark:convexifying_term}).
For this scenario, MOSEK returns status \texttt{UNKNOWN}; this indicates numerical issues, where the returned solution may or may not be close to a feasible and optimal solution \cite{mosekapsMOSEKOptimizationToolbox2025}. 
Here, the covariance dynamics are largely violated, according to a measure of loss in covariance propagation defined as follows:
\begin{equation}
	\text{Loss}_k = \norm{\phi(K_k, P_k) - P_{k+1}}_F / \norm{P_{k+1}}_F
\end{equation}
where \( K_k, P_k \) are the solution of the optimization, and $\phi(K_k, P_k)$ is the RHS of \cref{eq:full_covariance_propagation}. 
\cref{fig:example_cwh_covariance_propagation_loss} shows the loss for the two methods, showing the superior accuracy of the proposed method's covariance propagation; using tighter convergence tolerance for the proposed method can further improve the accuracy if desired. We also tested with different values of $\eta$ and alternative objective functions for the full-covariance approach, such as the EoQ form, \( \sum_k \norm{v_k}_2 + \lambda_{\max}(Y_k) + \eta \trace(Y_k)\), and \(\sum_k \norm{v_k}_2 + \trace(Y_k) \), all either resulting in numerical issues and/or lossy covariance propagation.
These results support our hypothesis: by working with the square root of the covariance matrix, solvers can work with better-scaled problems and are less likely to encounter numerical issues.
\section{Conclusions}
This paper presents a new solution method for chance-constrained covariance steering. The presented method addresses critical drawbacks in existing solution methods, namely: numerical ill-conditioning and cost formulation inflexibility in the full-covariance method as well as computational complexity in the block-matrix method. The key innovation lies in formulating the problem using the Cholesky factor of the covariance matrix, inspired by the literature on square root filtering. We then show the optimality properties of the proposed method; in the unconstrained case with expectation-of-quadratic cost, the solution is globally optimal, while in general the solution is locally optimal. The proof leverages techniques from manifold topology. Numerical examples support our theoretical claims of optimality and numerical reliability, and demonstrate scalability to large horizon sizes. For chance-constrained problems, the horizon length-scalability and optimality is competitive with the state-of-the-art.



\bibliographystyle{IEEEtran}
\bibliography{references}

\end{document}